\documentclass[11pt,epsfig]{amsart}

\usepackage{amsmath,amssymb,amscd,amsthm,mathrsfs}
\usepackage[all]{xy}
\usepackage{enumerate}
\usepackage[dvips]{graphicx}
\usepackage{stmaryrd}

\theoremstyle{plain}
\numberwithin{equation}{section}
\newtheorem{theorem}{Theorem}[section]
\newtheorem{proposition}[theorem]{Proposition}
\newtheorem{corollary}[theorem]{Corollary}
\newtheorem{lemma}[theorem]{Lemma}

\theoremstyle{definition}
\newtheorem{definition}[theorem]{Definition}

\newtheorem{remark}[theorem]{Remark}
\renewcommand{\proofname}{\textit{Proof}}

\def\rank{\mathop{\mathrm{rank}}\nolimits}
\def\dim{\mathop{\mathrm{dim}}\nolimits}

\def\Hom{\mathop{\mathrm{Hom}}\nolimits}
\def\hom{\mathop{\mathrm{hom}}\nolimits}
\def\Ext{\mathop{\mathrm{Ext}}\nolimits}

\def\<{{\langle}}
\def\>{{\rangle}}

\def\Aut{\mathop{\mathrm{Aut}}\nolimits}

\def\+{\mathop{\oplus}\nolimits}

\newcommand{\mf}[1]{{\mathfrak{#1}}}

\newcommand{\bb}[1]{{\mathbb{#1}}}
\newcommand{\mca}[1]{{\mathcal{#1}}}
\newcommand{\mr}[1]{{\mathrm{#1}}}

\title{Pure sheaves and Kleinian singularities}
\author{Kotaro Kawatani}
\date{\today}
\address{Department of Mathematics, 
Osaka University, 
Machikaneyama-cho, Toyonaka-city, Osaka, Japan}
\email{kawatani@math.sci.osaka-u.ac.jp}

\subjclass[2010]{Primary 14B05; Secondly 14F05, 17B22}

\begin{document}
\maketitle

\begin{abstract}
Grothendieck proved that any locally free sheaf on a projective line over a field (uniquely) decomposes into a direct sum of line bundles. 
Ishii and Uehara construct  an analogue of Grothendieck's theorem for pure sheaves on the fundamental cycle of the Kleinian singularity $A_n$. 
We first study the analogue for the other Kleinian singularities. 
We also study the classification of rigid pure sheaves on the reduced scheme of the fundamental cycles. 
The classification is related to the classification of spherical objects in a certain Calabi-Yau $2$-dimensional category. 
\end{abstract}

\section{Introduction}\label{intro}

Let $\mca E$ be a locally free sheaf on a projective line $\bb P^1$ over a field $k$. 
As was proven by Grothendieck \cite{Gr}, the sheaf $\mca E$ decomposes into a direct sum of line bundles on $\bb P^1$ and the decomposition is unique up to isomorphisms.  
Hence we have a complete classification not only of locally free sheaves but also of indecomposable sheaves on $\bb P^1$.

It may be natural to study an analogue of Grothendieck's theorem for $\bb P^n$, but it seems difficult. 
In fact, if $n >1$ then there exists an indecomposable locally free sheaf on $\bb P^n$ whose rank is greater than $1$. 
A simple example of such a sheaf is the tangent sheaf on $\bb P^n$. 
Moreover the classification of indecomposable locally free sheaves is more difficult in the case of lower rank (cf. \cite{Har79}).

Though higher dimensional analogue of Grothendieck's theorem is difficult, Ishii and Uehara prove a beautiful analogue for the fundamental cycle $Z_{A_n}$ of the Kleinian singularity $A_n$. 
To recall their result, a non-zero sheaf $\mca F$ on a scheme $Y$ is said to be \textit{pure} if the support of any non-trivial subsheaf of $\mca F$ has the same dimension of $Y$. 
We note that if $Y$ is smooth and $1$-dimensional then a pure sheaf on $Y$ is equivalent to a locally free sheaf on $Y$. 
Thus a pure sheaf is a natural generalization of locally free sheaves for reducible schemes such as $Z_{A_n}$. 
Ishii and Uehara prove the following:

\begin{theorem}[{\cite[Lemma 6.1]{IU}}]\label{IUthm}
Let $\mca E$ be a pure sheaf on $Z_{A_n}$. 
Then $\mca E$ decomposes into a direct sum of invertible sheaves on connected subtrees of $Z_{A_n}$. 
Moreover, the decomposition is unique up to isomorphisms. 
\end{theorem}

We first study an analogue of Theorem \ref{IUthm}.

\begin{theorem}[=Corollary \ref{maincor}]\label{main1}
Let $Z$ be the fundamental cycle of a Kleinian singularity except for $A_n$. 
Then 
\[
\max \{ \rank_Z \mca E  \mid  \mca E\mbox{ is an indecomposable pure sheaf on }Z \} = \infty. 
\]
\end{theorem}
We remark that the usual rank of sheaves is not appropriate since our scheme is reducible. 
Thus we introduce more suitable ``rank" of sheaves in our setting (see Definition \ref{rank}). 
By using it the first half of Theorem \ref{IUthm} can be restated that the maximum of the rank of indecomposable pure sheaves is $1$. 
It may be natural to expect that the maximum of the rank of indecomposable pure sheaves are bounded for the other Kleinian singularities. 
Our theorem gives an counter-example of the expectation.

The second aim of this note is to study the classification of ``$\mca O_X$-rigid" pure sheaves on $Z$. 
The classification is related to the classification of spherical objects in a certain category (for the definition of spherical objects, see also \cite{Huy06} or \cite{ST}).

To explain the relation, 
let $ X $ be the minimal resolution of a Kleinian singularity. 
It is well-known that the fundamental cycle $Z$ of the singularity is the schematic fiber of the singularity by the resolution. 
Since $Z$ is a subscheme of $X$, we have a natural embedding $\iota \colon Z \to X$. 
We denote by $D_Z(X)$ the bounded derived category of coherent sheaves on $X$ supported on $Z$. 
A coherent sheaf $\mca E$ on $Z$ is said to be \textit{$\mca O_X$-rigid} if the push forward $\iota _* \mca E$ by $\iota $ is rigid, that is, $\Ext_X^1 (\iota _* \mca E, \iota _* \mca E) =0$.

Ishii and Uehara show that each cohomology (with respect to the standard $t$-structure) of spherical objects in $D_Z(X)$ is the push forward $\iota _* \mca E$ of a pure sheaf $\mca E $ on $Z$ which is $\mca O_X$-rigid. 
If the singularity is $A_n$, then the classification of $\mca O_X$-rigid sheaves is a direct consequence of Theorem \ref{IUthm}. 
By using the classification, Ishii and Uehara classify spherical objects in $D_Z(X)$ for the Kleinian singularity $A_n$ (the details are in \cite[Proposition 1.6]{IU}). 
One might hope to classify spherical objects for the other Kleinian singularities following Ihii and Uehara's approach, by the first classifying indecomposable pure $\mca O_X$-rigid sheaves. 
Theorem \ref{main1} is an evidence that this likely to be a rather difficult problem and we do not solve it in this paper. 
However we do prove the following result, which leaves hope that such a classification might be achieved in the future.

\begin{theorem}\label{mainthm2}
Let $\mca E$ be an indecomposable pure sheaf on the reduced scheme $Z_{r}$ of the fundamental cycle of a Kleinian singularity except for $A_n$. 
If $\mca E$ is $\mca O_X$-rigid, then $\rank_{Z_r} \mca E \leq 3$ and the inequality is best possible. 
\end{theorem}

The proof of Theorem \ref{mainthm2} will be postponed till the end of Section \ref{4}. 
The essential part is in the proof of Propositions \ref{mainD} and \ref{bestD}.

\subsection*{Acknowledgement}
The author thanks the referee for valuable comments which simplify the proof of Theorem \ref{mainthm1} and improve readability. 
He is supported by JSPS KAKENHI Grant Number JP 16H06337. 
\section{Notations and Conventions}

Throughout this note, our field $k$ is algebraically closed and Kleinian singularities are given by $\mr{Spec}\,  k\llbracket x,y,z \rrbracket/ f(x,y,z)$ where $f(x,y,z)$ is one of the following:
\begin{center}
\begin{tabular}{cll}
$A_n$	&	$x^2 + y ^2 +z^{n+1}$ & for $n \geq 1$\\
$D_n$	&	 $x^2 + y^2 z + z^{n-1}$ & for $n \geq 4$\\
$E_6$	&	 $x^2 + y^3 + z^4$\\
$E_7$	&	$x^2 + y^3 + yz ^3$\\
$E_8$	&	$x^2 + y^3 + z^5$. 
\end{tabular}
\end{center}

Let $Z$ be the fundamental cycle of the singularity $D_{n}$. 
The $i$-th irreducible component $C_i$ of $Z$ is denoted as in Figure \ref{tree}. 
Then it is well-known that $Z$ is $C_1 +C_2 + \sum_{i=3}^{n-1}2 C_i +C_{n}$. 
Similarly the $j$-th irreducible component $C_j$ of the fundamental cycle of the singularities $E_{6}$, $E_7$ or $E_8$ is denoted as in Figure \ref{treeE}.

\begin{figure}[htb]
\begin{minipage}{0.49\hsize}
\begin{center}
\includegraphics[clip, width=67mm]{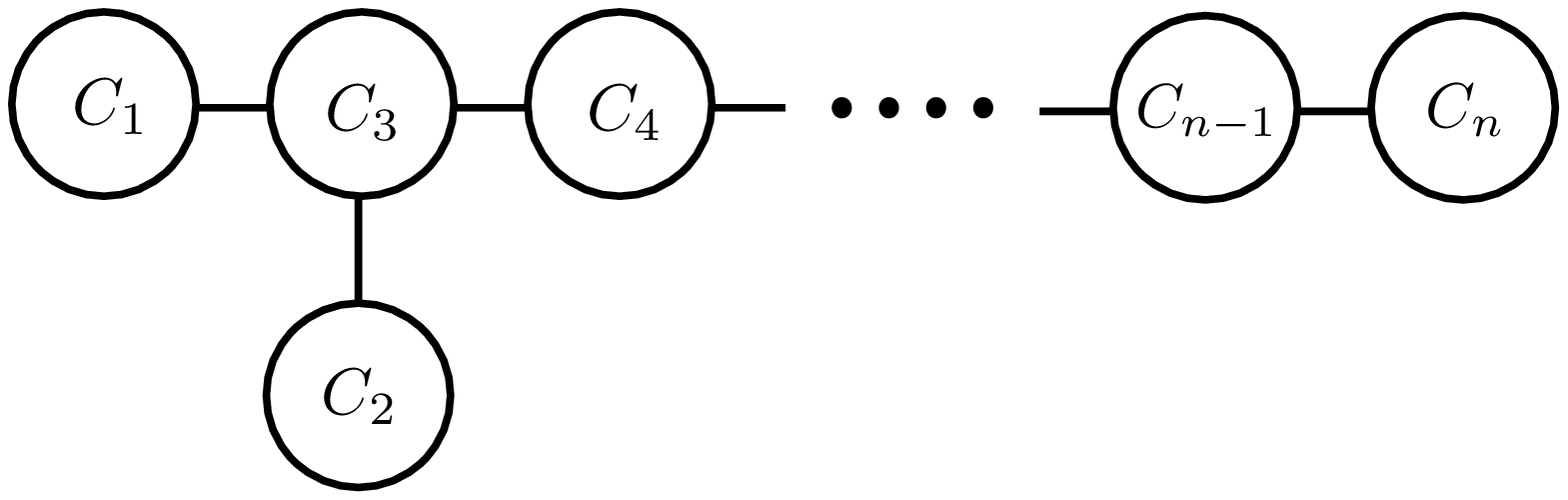}
\end{center}
\caption{}
\label{tree}
\end{minipage}
\begin{minipage}{0.49\hsize}
	\begin{center}
\includegraphics[clip, width=74mm]{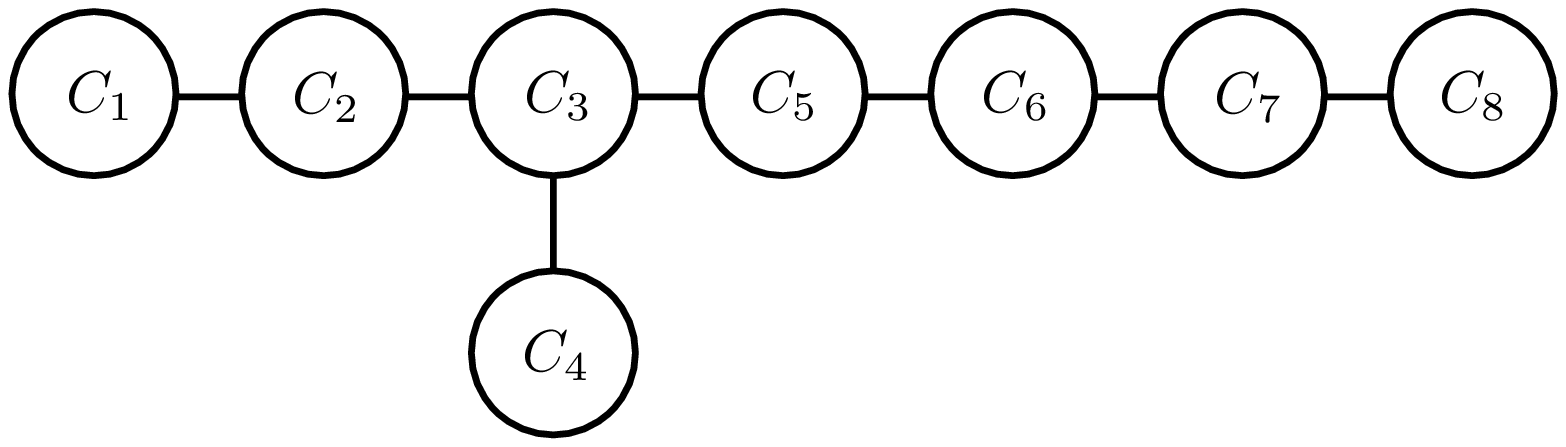}
	\end{center}
\caption{}
\label{treeE}
\end{minipage}
\end{figure}
\begin{remark}\label{change index}
We note that the chain $\sum_{i=2}^5 C_i$ in Figure \ref{treeE} gives the reduced scheme of the fundamental cycle of the singularity $D_4$. We use this identification in the proof of Proposition \ref{propE}. 
\end{remark}

Let $\mca D$ be a $k$-linear triangulated category. 
We denote by $\hom^p(E,F)$ the dimension of the vector space $\Hom^p_{\mca D}(E,F) =  \Hom_\mca D(E,F[p])$ for $E$ and $F \in \mca D$. 
The category $\mca D$ is said to be of finite type if the sum $\sum_{p \in \bb Z} \hom^p(E,F)$ is finite for any $E,F \in \mca D$. 
If $\mca D$ is of finite type 
then the Euler characteristic
\[
\chi(E,F)= \sum_{p \in \bb Z} (-1)^p \hom^p(E, F)
\] 
is well-defined. 
If the Serre functor of $\mca D$ is isomorphic to the double shift $[2]$, then $\mca D$ is said to be \textit{$2$-dimensional Calabi-Yau} (for simplicity CY2).  
If $\mca D$ is CY2, then we have $\hom^p(E,F) = \hom^{2-p}(F,E)$. 
One of the best example of CY2 categories is $D_Z(X)$.

\section{Indecomposable pure sheaves}\label{2}

\begin{definition}\label{rank}
Let $Z'$ a $1$-dimensional closed subscheme of the fundamental cycle $Z$ of a Klein singularity and $\iota' \colon Z' \to  X$ be the embedding to the minimal resolution $X$ of the singularity. 
We define the rank of a sheaf $\mca E$ on $Z'$ as follows: 
\[
\rank_{Z'} \mca E := \min \{ a \in \bb Z_{\geq 0} \mid c_1(\iota' _* \mca E) \leq a \cdot Z'    \}, 
\]
where $c_1$ is the first Chern class. 
\end{definition}

\begin{remark}
We would like to define a rank so that the structure sheaf of any $Z'$ has rank $1$. 
One of a naive generalization of the usual rank is the following: 
The rank of a sheaf $\mca E$ on $Z'$ is the maximum rank of $\mca E$ on each irreducible components of $Z'$. 
Such a generalization does not satisfy our requirement if $Z'$ is the fundamental cycle $Z$ of a Kleinian singularity except for $A_n$. 
\end{remark}

By using Definition \ref{rank} the first half of Theorem \ref{IUthm} can be restated as follows
\[
\max \{ \rank_{Z_{A_n}} \mca E  \mid \mca E\mbox{ is an indecomposable pure sheaf on } Z_{A_n} \} = 1. 
\] 
Contrary to the singularity $A_n$, we prove the following for the singularity $D_4$:

\begin{theorem}\label{mainthm1}
Let $\mca E$ be  a pure sheaf on the reduced scheme $Z_{r}$ of the fundamental cycle $Z$ of the singularity $D_4$. 
Then there exists an indecomposable pure sheaf $\mca E$ on $Z_{r}$ of $\rank_{Z_{r}} \mca E=r$ for any $r \in \bb N$. 
\end{theorem}

Before the proof we denote by $\mca O_{C_1+C_2+C_3}(a_1,a_2,a_3)$ an invertible sheaf on the chain $C_1 + C_2 +C_3$ such that the degree on each irreducible component $C_i$ is $a_i$.

\begin{remark}
A key ingredient of the proof of Theorem \ref{mainthm1} is a choice of particular sheaves $\mca L_n$, where $\mca L_n = \mca O_{C_1 + C_2 + C_3}(n,-n,0)$ for $n \in \bb Z$.  
Any pair $(\mca L_n, \mca L_m)$ (for $n \neq m$) has the following property: 
For any morphism $ f \colon \mca L_n \to \mca L_m$, the induces morphism $ f_* \colon \Ext^1_{Z_r} (\mca O_{C_4} , \mca L_n) \to \Ext^1_{Z_r}(\mca O_{C_4}, \mca L_m)$ is zero (the details are in Lemma \ref{ferox}). 
If the singularity is $A_n$, such a pair does not exist. 
\end{remark}

\begin{proof}
Take an invertible sheaf $\mca L_n = \mca O_{C_ 1 + C_2 + C_3}(n,-n,0)$ on $C_1+C_2+C_3$ for an integer $n \in \bb Z$. 
It is easy to see $\Ext_{Z_{r}}^1(\mca O_{C_4}, \mca L_n) \cong \Ext_{Z_{r}}^1 (\mca O_{C_4}, 
\mca O_{C_3}) $. 
We wish to describe $\Ext^1_{Z_{r}} (\mca O_{C_4}, \mca O_{C_3}) $ in an explicit way. 
By the locally free resolution of $\mca O_{C_4}$ as $\mca O_{Z_{r}}$-module, 
we see $\mca Hom^0_{Z_{r}} (\mca O_{C_4}, \mca O_{C_3} ) =0$ and  
$\mca Ext^1_{Z_{r}}(\mca O_{C_4}, \mca O_{C_3}) \cong k(x)$ where $x \in C_3 \cap C_4$. 
Thus we have $\Ext_{Z_{r}}^1(\mca O_{C_4}, \mca L_n) \cong \Ext^1_{Z_{r}} (\mca O_{C_4}, \mca O_{C_3})  \cong H^0\big(k(x)\big)
$ by the local-to-global spectral sequence.

Let $I$ be an arbitrary finite subset of $\bb Z$. 
The vector space $\bigoplus _{n \in I} \Ext_{Z_{r}}^1 (\mca O_{C_4}, \mca L_n)$ is denoted by $V_I$. 
Any extension class $[\mca E] \in V_I$ can be identified with a column vector with respect to a natural basis of $V_I$. 
Take the universal extension $[\mca U_I]$, that is, $[\mca U_I]$ is a vector whose components are all $1$. 
We wish to prove that $\mca U_{I} $ is indecomposable.

Suppose to the contrary that $\mca U_{I}$ decomposes into $\mca F \oplus \mca G$.
We can assume $\mr{Supp}\ \mca G \supset C_4$ without loss of generality. 
Then we have $\Hom_{Z_r}(\mca F, \mca O_{C_4})=0$ since $\mr{Supp}\ \mca F \subset C_1 + C_2 +C_3$. 
Hence the natural morphism $ f \colon \mca F \oplus \mca G \to \mca O_{C_4}$ splits into $0 \oplus \tilde f$ where $0$ is the zero morphism from $\mca F$ and $\tilde f \colon \mca G \to \mca O_{C_4}$. 

Let $K$ be the kernel of the morphism $\tilde f$. 
Then we have
\[
\mca F \oplus K \cong \bigoplus_{n \in I}\mca L_n. 
\]
and see that $\mca F \oplus K $ is a pure $\mca O_{C_1 +C _2 + C_3}$-module. 
Thus, by Theorem \ref{IUthm}, we see
$\mca F \cong \bigoplus _{n_i \in I'}\mca L_{n_i}$ and $K \cong \bigoplus_{n_j \in I''} \mca L_{n_j}$ where $I' \coprod I'' = I$. 
In particular we have the following diagram of distinguished triangles:
\[
\xymatrix{
\mca F \oplus \mca G	\ar[d]\ar[r]^f	&\mca O_{C_4}	\ar[d]\ar[r]^-{u} &\bigoplus_{n \in I} \mca L_n [1]\ar[d]^{\pi[1]} \\
\mca F		\ar[r]				&0				\ar[r]		&\bigoplus _{ n_i \in I'}\mca 
L_{n_i}[1].	\\
}
\]
Hence the composite $\pi[1] \circ u$ is zero in the derived category $D(Z_r)$ on $Z_r$. 
Thus the corresponding coefficients $[\mca U_{I}] \in \Ext^1_{Z_{r}}(\mca O_{C_4}, \mca F) 
\oplus \Ext^1_{Z_{r}}(\mca O_{C_4}, \mca G)$ to $\mca F$ should be $0$. 
Moreover we see that the natural representation of the automorphism group $\Aut (\bigoplus \mca L_n)$ on $V_I$ is contained in diagonal matrices by Lemma \ref{ferox} (below). 
This contradicts the definition of $\mca U_I$. 
\end{proof}

\begin{lemma}\label{ferox}
We denote by $\mca L_{i}$ a pure sheaf $\mca O_{C_1 +C_2 + C_3 } (i,-i,0)$ for an integer $i$. 
For any finite subset $I \subset \bb Z$, the vector space $\bigoplus _{i \in I} \Ext^1_{Z_{r}}(\mca O_{C_4}, \mca L_{i})$ is denoted by $V_I$. 
Then the image of a natural representation 
\[
\rho \colon \mr{End}_{Z_{r}} \big(\bigoplus _{i \in I} \mca L_i \big) \to \mr{End} _{k}(V_I )
\]
 is contained in diagonal matrices with respect to a natural basis of $V_I$. 
\end{lemma}

\begin{proof}
Let $\mca D$ be the derived category on $Z_r$. 
The vector space $\mr{End}_{Z_{r}} \big(\bigoplus _{i \in I} \mca L_i \big) $ decomposes into as follows:
\[
\mr{End}_{Z_{r}} \big(\bigoplus _{i \in I} \mca L_i \big) 
 \cong \bigoplus _{i,j \in I} \Hom_{Z_r} ( \mca O_{C_1+C_2+C_3}, \mca L_{j-i} ). 
\]
By the symmetry for $C_1$ and $C_2$ we can assume $ \ell  = i - j \geq 0$.

If $-\ell < 0$ then $H^0(\mca O_{C_2+ C_3}(-\ell,0))$ is zero. 
Hence the natural inclusion $\mca O_{C_1}(\ell-1) \to \mca L_{\ell}$ induces an isomorphism 
\begin{equation}
\Hom(\mca O_{C_1 + C_2 +C_3}, \mca L_{\ell}) 
 \cong  \Hom(\mca O_{C_1 + C_2 +C_3}, \mca O_{C_1}(\ell-1)) .  \label{hung}
\end{equation}
Thus any morphism $ \varphi \in \Hom(\mca O_{C_1 + C_2 +C_3}, \mca L_{\ell}) $ factors through $\mca O_{C_1}(\ell-1)$ and 
hence the induced morphism in $\mca D$
\[
\varphi _* \colon \Hom_{\mca D}^0(\mca O_{C_4}[-1], \mca O_{C_1 +C_2 + C_3}) \to \Hom_{\mca D}^0(\mca O_{C_4}[-1], \mca L_{\ell})
\]
 factors through $\mca O_{C_1}(\ell-1)$. 
Thus the morphism $\varphi_* $ should be zero since the intersection $C_1 \cap C_4$ is empty. 
Hence the action of $\mr{End}_{Z_{r}} (\bigoplus_{i \in I} \mca L_i)$ is contained in diagonal component of $\mr{End}_{k}(V_I)$. 
\end{proof}

\begin{corollary}\label{maincor}
Let $Z$ be the fundamental cycle of a Kleinian singularity except for $A_n$. 
Then there is an indecomposable pure sheaf on $Z$ of rank $r$ for any $r \in \bb N$. 
In particular the following holds:
\[
\max \{ \rank_Z \mca E \mid \mca E \mbox{ is an indecomposable pure sheaf on }Z  \} = \infty. 
\]
\end{corollary}

\begin{proof}
Let $Z_{4,r}$ be the reduced scheme of the fundamental cycle of the singularity $D_4$. 
Then $Z_{4,r}$ is a closed subscheme of $Z$.

As in the proof of Theorem \ref{mainthm1}, the universal extension $\mca U_I$ of $\Ext^1_{Z_{4,r}} (\mca O_{C_4}, \bigoplus_{n \in I} \mca L_n)$ is indecomposable pure sheaves on $Z_{4,r}$. 
The push forward $\iota _* \mca U$ by the closed embedding $\iota \colon Z_{4,r} \to Z$ is also a pure sheaf on $Z$. 
Moreover the push forward $\iota _*$ is a fully faithful functor from $\mr{Coh} (Z_{4,r})$ to $\mr{Coh}(Z)$ and the full subcategory $\iota_*\mr{Coh} (Z_{4,r})$ is closed under direct summands. 
Thus the assertion holds. 
\end{proof}

\section{$\mca O_X$-rigid pure sheaves on $D_n$}\label{3}

For any closed embedding $f \colon Z \to X$, the push forward $f_* \colon \mr{Coh}(Z) \to \mr{Coh}(X)$ is fully faithful, but the derived push forward $f_* \colon D(Z) \to D(X)$ is not. 
To analyze the difference the following lemma is necessary.

\begin{lemma}\label{extcomp}
Let $f \colon Z \to X$ be a closed embedding of $Z$ to an algebraic variety $X$. 
Let $\mca F$ and $\mca E$ be sheaves on $Z$.  
The canonical map 
$\Ext^1_Z(\mca F , \mca E) \rightarrow\Ext^1_X(f_* \mca F, f_* \mca E)$ is injective. 
\end{lemma}

\begin{proof}
By the adjunction we have $\Ext^p_X(f_* \mca F, f_* \mca E) \cong \Ext^p_Z\big(\bb L f^*  (f_* \mca F), \mca E
\big)$ (note that $f$ is affine morphism). 
By the canonical morphism 
$\bb L f^* f_* \mca F \to \mca F$ we have the following distinguished triangle in the derived category $\mca D$ of coherent sheaves on  $Z$:
\[
\begin{CD}
\mf F @>>>	\bb Lf^* f_* \mca F @>>> \mca F @>>> \mf F[1]. 
\end{CD}
\]
Since $\bb L^0 f^* f_* \mca F  = f^* f_*\mca F\cong \mca F$, we see that the $p$-th cohomology (with respect to the standard $t$-structure) of the complex $\mf F $ vanishes for $p\in \bb Z_{\geq 0}$. 
Hence we have $\Hom_Z^q(\mf F, \mca E)=0$ for $q \in \bb Z_{\leq 0}$.

By taking $\bb R \Hom_{\mca D} (-, \mca E)$ to the above sequence we have the following exact sequence:
\[
\begin{CD}
\Hom^0_{\mca D}(\mf F, \mca E)	@>>>	\Hom^1_{\mca D}(\mca F, \mca E)	@>\kappa>>	\Hom^1_{\mca D}(\bb Lf^*f_* \mca F, \mca 
E )	@>>> \Hom^1_{\mca D}(\mf F, \mca E)
\end{CD}. 
\]
Note that the canonical morphism $\Ext^1_Z(\mca F, \mca E) \to \Ext^1_X(f_* \mca F, f_* \mca E)$ is given by $\kappa$. 
Since $\Hom^0_Z(\mf F, \mca E)=0$ the morphism $\kappa$ is injective. 
\end{proof}

\begin{corollary}\label{isom}
Let $Z$ be a chain of rational curves in $X$ and let $\{ C_i \}_{i=1}^{n}$ be a set of irreducible components of $Z$. 
Then for $i \neq j$, we have
\[
\Ext^1_Z(\mca O_{C_i}(d_i), \mca O_{C_j}(d_j)) \cong \Ext^1_X(f_* \mca O_{C_i}(d_i),  f_* \mca O_{C_j}(d_j)). 
\]
\end{corollary}

\begin{proof}
A locally free resolution of $f_*\mca O_{C_i}(d_i)$ is given by 
\[
\begin{CD}
  0	@>>>	\mca O_X(D-C_i) @>>> \mca O_X(D) @>>> 0
\end{CD},  
\]
where $D$ is a divisor on $X$ such that $D.C_i =d_i$. 
Thus $\bb Lf^* f_* \mca O_{C_i}(d_i)$ is given by the following:
\[
\begin{CD}
  0	@>>>	\mca O_Z(D-C_i) @>\delta>> \mca O_Z(D) @>>> 0
\end{CD}. 
\]
In particular $\bb L^{-1} f^{*}f_* \mca O_{C_i}(d_i)$ is the kernel of $\delta$ which is isomorphic to 
$\mca O_{C_i}(D-Z)$. 
Moreover $\mf F$ is isomorphic to $\mca O_{C_i}(D- Z)[1]$. 
Since $C_i \neq C_j$ we have 
\[
\Hom^1_Z(\mf F, \mca O_{C_j}(d_j)) = \Hom^0_Z(\mca O_{C_i}(D- Z), \mca O_{C_j}(d_j)) =0. 
\]
This is the desired conclusion. 
\end{proof}

We first introduce a relation on a collection of sheaves and secondly prove that the relation defines an order:

\begin{definition}\label{Order}
Let $\{ \mca N_i \}_{i \in I}$ and $\{ \mca L_j \}_{j \in J}$ be finite collections of isomorphism classes of sheaves on a scheme $Y$. 
Suppose that $\mf N$ and $\mf L$ satisfy the following:
\begin{enumerate}
\item[$(a)$] Endomorphism rings $\mr{End}_Y(\mca L_j)$ and $\mr{End}_Y(\mca N_i)$ are generated by the identity for each $i$ and $j$. 
\item[$(b)$] Each pair $(\mca L_j, \mca N_i)$ satisfies $ \dim \Ext^1_Y(\mca L_j, \mca N_i)=1$. 
\end{enumerate}

(1) If there is a morphism $f\colon \mca L_{j_2} \to \mca L_{j_1}$ such that $f^*\colon \Ext^1_Z(\mca L_{j_1}, \mca N_i) \to \Ext^1_Z(\mca L_{j_2}, \mca N_i)$ is nonzero for all $i \in I$ then we define a relation $\mca L_{j _1} \leq  \mca L_{j_2} $ on $\{ \mca L_j  \}_{j \in J}$. 

(2) Dually if there is a morphism $g\colon \mca N_{i_1} \to \mca N_{i_2}$ such that the induced morphism $g_* \colon \Ext^1_Z(\mca L_j, \mca N_{i_1}) \to \Ext^1_Z(\mca L_j, \mca N_{i_2}) $ is nonzero for all $ j \in J$ then we define a relation $\mca N_{i_1} \leq  \mca N_{i_2}$ on $\{ \mca N_i  \}_{i \in I}$. 
\end{definition}

%
%
%
%
%

\begin{proposition}\label{Order2}
The relations on $\{ \mca N_i \}_{i \in I}$ and $\{ \mca L_j \}_{j \in J}$ in Definition \ref{Order} respectively define orders. 
In particular both are posets. 
\end{proposition}

\begin{proof}
Since the proof is similar, it is enough to show the claim for $\{ \mca {L}_j \}_{j \in J}$. 
The reflexivity is obvious since the identity gives the identity on $\Ext^1_{Z_r} (\mca L_j, \mca N_i)$.  

Suppose $\mca L_{j_1} \leq \mca L_{j_2}$ and $\mca L_{j_2} \leq \mca L_{j_1}$. 
Then there exist morphisms $f_1: \mca L_{j_2} \to \mca L_{j_1}$ and $f_2: \mca L_{j_1} \to \mca L_{j_2}$. 
By the condition $(b)$ in Definition \ref{Order}, both $f_1^*$ and $f_2^*$ are isomorphisms. 
In particular the compositions $(f_1 \circ f_2)^*$ and $(f_2 \circ f_1)^*$ are nonzero morphisms. 
Thus two morphisms $f_1 \circ f_2 \in \mr{End}(\mca L_{j_1})$ and $f_2 \circ f_1 \in \mr{End}(\mca L_{j_2})$ are not zero. 
By the condition $(a)$ in Definition \ref{Order}, we see that $f_1 \circ f_2$ and $f_2 \circ f_1$ are identities up to scalar and hence $\mca L_{j_1} \cong \mca L_{j_2}$. 

For the transitivity let us suppose $\mca L_{j_1} \leq \mca L_{j_2}$ and $\mca L_{j_2} \leq \mca L_{j_3}$. 
Similarly as above the composition $f_1 \circ f_2$ of two morphisms $f_1: \mca L_{j_2} \to \mca L_{j_1} $ and $f_2 : \mca L_{j_3} \to \mca L_{j_2}$ induces a non-zero morphism $(f_1 \circ f_2 )_* \colon \Ext^1_Z (\mca L_{j_3}, \mca N_i) \to \Ext^1_Z(\mca L_{j_1}, \mca N_i)$. 
Thus we obtain $\mca L_{j_1 } \leq \mca L_{j_3}$
\end{proof}

\begin{remark}
We are interested in the classification of $\mca O_X$-rigid pure sheaves and study the classification in Lemma \ref{minimal} and Proposition \ref{propE}. 
Any pure sheaf $\mca E$ on $Z$ of a Kleinian singularity is given by a successive extension of pure sheaves on subtrees (see the filtration (\ref{horoyoi}) below). 
The poset structure defined in Proposition \ref{Order2} is convenient to analyze the successive extension. 
%
%
%
%
%
\end{remark}

We are ready to prove our main proposition in this section.

\begin{proposition}\label{mainD}
Let $X$ be the minimal resolution of the singularity $D_n$ and $Z_r$ be the reduced scheme of the fundamental cycle $Z$ of $D_n$. 
Suppose that $\mca E$ is an indecomposable pure sheaf on $Z_r$. 
If $\mca E$ is $\mca O_X$-rigid then we have $\rank_{Z_r} \mca E \leq 3$. 
\end{proposition}

\begin{proof}
We have $Z_{r} =\sum_{i=1}^{n} C_i $ by the definition. 
Take a pure sheaf $\mca E$ on $Z_r$ which is not necessary indecomposable but $\mca O_X$-rigid. 
We shall show that the rank of each direct summand of $\mca E$ is at most $3$.

Let $\mca F$ be the kernel of the restriction $\mca E \to \mca E \otimes \mca O_{C_4+ \cdots +C_n}$. 
By taking saturation if necessary, we can assume that the sheaf $\mca E$ fits into the short exact sequence 
\[
\begin{CD}
0 @>>> \mca F @>>> \mca E @>>> \mca G @>>> 0
\end{CD}
\]
where $\mca F$ is a pure sheaf on $C_1 + C_ 2 + C_3$ and $\mca G$ is a pure sheaf on $C_4+\cdots + C_n$. 
Both $\mca F$ and $\mca G$ are respectively direct sums of invertible sheaves of subtrees by Theorem \ref{IUthm} since both trees $C_1 +C_2 + C_3$ and $C_4 + \cdots  +C_n$ is the fundamental cycle of respectively $A_3$ and $A_{n-3}$:
\[
\mca F = \bigoplus_{i \in I} \mca N_{i } \mbox{ and }\mca G = \bigoplus _{j \in J} \mca L_j. 
\]
Without loss of generality we can assume the following
\begin{itemize}
\item The support of each $\mca N_i$ contains $C_3$. 
\item The support of each $\mca L_j$ contains $C_4$ and is connected. 
\end{itemize}

Now we claim the following:

%

\begin{lemma}\label{minimal}
Let $\mf N$ be the collection of direct summands $\{ \mca N_{i} \}_{i \in I}$ of $\mca F$ and $\mf L$ the collection of direct summands $ \{ \mca L_{j}  \} _{j \in J}$ of $\mca G$. 
\begin{enumerate}
\item[$(1)$] Both $\mf N$ and $ \mf L$ are posets with respect to the relation in Definition \ref{Order}.  
\item[$(2)$] There exist at most $3$ minimal elements in any subposet of $\mf N$ and the poset $\mf L$ is totally ordered. 
\end{enumerate}
\end{lemma}

Before the proof of Lemma \ref{minimal}, we finish the proof.  
Note that $\mca E$ defines a class $[\mca E ]$ in $\bigoplus_{i \in I, j \in J} \Ext^1_{Z_{r}}(\mca L_j, \mca N_i)$ denoted by $V_{IJ}$. 
Put $m=\# I$ and $n = \# J$. 
Clearly $V_{IJ}$ can be identified with the set of $m\times n$ matrices and we can write $[\mca E]$ by a matrix 
\[
[\mca E] =\begin{pmatrix} 
e_{11} & e_{12} & \cdots & e_{1n} \\
e_{21} & e_{22} & \cdots & e_{2n} \\
\vdots & \vdots		&		& \vdots \\
e_{m1} & e_{m2} & \cdots & e_{mn} 
\end{pmatrix}. 
\]

If $\mca N_{i_1} \leq  \mca N_{i_2}$ and $e_{i_1, j} \neq 0$ then 
we can assume $e_{i_2,j}=0$ by row fundamental transformations induced by a morphism $\mca N_{ i_1} \to \mca N_{i_2}$. 
Since $\mf {N}$ has at most $3$ minimal elements by Lemma \ref{minimal}, there exists at most $3$ components $e_{ij}$ such that $e_{ij}=1$ in each row. 

Similarly if $\mca L _{j_1} \geq  \mca L_{j_2}$ and $e_{i, j_1} \neq 0$
then we can assume $e_{i, j_2}=0$ by column fundamental transformations. 
Since $\mf L$ is totally ordered by Lemma \ref{minimal}, there exists at most $1$ components $e_{i'j'}$ such that $e_{i'j'}=1$ in each column. 
This means that the rank of each direct summand of $\mca E$ is at most $3$ since $Z_r$ is reduced. 
\end{proof}

\renewcommand{\proofname}{\textit{Proof} of Lemma \ref{minimal}}
\begin{proof}
Let $\iota \colon  Z_{r}  \to X$ be the embedding to the minimal resolution of the singularity. 
Since $\Hom_{Z_r}( \mca F, \mca G)$ is zero the push forwards $\iota _* F$ and $ \iota _* \mca G$ are rigid by \cite[Lemma 2.5]{BB13}. 
Each direct summand of $\mca G$ is an invertible sheaf on a connected subtree of $C_4 + \cdots + C_n$. 
Then the order introduced in \cite[Section 6.1]{IU} gives the order in Definition \ref{Order}. 
In particular $\mf {L}$ is totally ordered. 


To determine $\mf N$, similarly as before, take a filtration of $\mca F$ 
\begin{equation}
0 \subset \mca F_1 \subset \mca F _3 \subset \mca F_2 = \mca F \label{horoyoi}
\end{equation}
such that 
\begin{itemize}
\item $\mca F_1$ and $\mca F_3$ are pure on respectively $C_1$ and $C_1 + C_3$,
\item $\mca F_3/ \mca F_1$ is pure on $C_3$ and  
\item $\mca F_2/ \mca F_3$ is a pure sheaf on $C_2$. 
\end{itemize}

Similarly, $\mca F_3/ \mca F_1$ and $\mca F_2/ \mca F_3$ are rigid by Lemma \ref{extcomp} and \cite[Lemma 2.5]{BB13}. 
Since $\mca F_1$ is rigid and pure, there exists an integer $a_1$ such that $\mca F_1= \mca O_{C_1}(a_1)^{\oplus m_1} \oplus  \mca O_{C_1}(a_1+1)^{\oplus n_1}$. 
A similar statement applies to $\mca F_3 /\mca F_1$ amd $\mca F_2 / \mca F_3$. 
So there are three integers $\{ a_1, a_2 , a_3  \}$ such that every summand of $\mca F$ is one of the $18$ possibilities listed in Table \ref{Summand}. 
%
%
%
%
%
%
\begin{table}[htbp]
\begin{center}
\resizebox{1\hsize}{!}{
\begin{tabular}{|c|c|c|}
\hline
$\mca O_{C_1+C_3}(a_1+1,a_3)$	&$\mca O_{C_1+C_2+C_3}(a_1+1, a_2, a_3+1)$	&$\mca O_{C_1+C_2+C_3}(a_1+1, a_2+1, a_3+1)$		\\
\hline
$ \mca O_{C_1+C_3}(a_1+2 , a_3)$	&$\mca O_{C_1+C_2+C_3}(a_1+2, a_2, a_3 +1)$	&$\mca O_{C_1+C_2+C_3}(a_1+2, a_2+1, a_3+1)$	\\
\hline
$\mca O_{C_3}(a_3)$	&$\mca O_{C_2+C_3}(a_2, a_3+1)	$	&$\mca O_{C_2 +C_3}(a_2+1, a_3+1)$	\\
\hline
$ \mca O_{C_1+C_3}(a_1 +1, a_3+1)$	&$\mca O_{C_1+C_2+C_3}(a_1+1, a_2, a_3+2)$	&$\mca O_{C_1+C_2+C_3} (a_1 +1, a_2 +1, a_3+2)$	\\
\hline
$\mca O_{C_1+C_3}(a_1 +2, a_3+1)$	&$\mca O_{C_1+C_2+C_3}(a_1+2, a_2,a_3+2)$	&$\mca O_{C_1+C_2+C_3}(a_1+2, a_2+1, a_3+2)$	\\
\hline
$\mca O_{C_3}(a_3 +1)$	&$\mca O_{C_2+C_3}(a_2, a_3+2)$		&$\mca O_{C_2+C_3}(a_2+1, a_3+2)$\\
\hline
\end{tabular}
}
\end{center}
\caption{We denote by $\mca N_{ij}$ the $i$-th column and $j$-the row component in the table. For instance, $\mca N_{31}=\mca O_{C_3}(a_3)$. }\label{Summand}
\end{table}%

We prove the first assertion. 
Recall that $\Ext^1_{Z_{r}}(\mca L_j, \mca N_i) $ is isomorphic to $H^0(\mca O_{x})$ where $x \in C_3 \cap C_4$. 
Since the support of each $\mca N_i$ contains $C_3$, both $\mf N$ and $\mf L$ satisfy the conditions $(a)$ and $(b)$ in Definition \ref{Order}. 
If $\Hom(\mca O_{C_4}(d), \mca O_{C_4}(d'))$ is not zero, where $d$ and $d' \in \{ a_4, a_4+1  \}$, then there exists a morphism $\psi \colon \mca O_{C_4}(d) \to \mca O_{C_4}(d')$ which induces a non-zero morphism on $H^0(\mca O_x)$.  
Similarly, if $\Hom(\mca N_{i_1}, \mca N_{i_2})$ is not zero for $\mca N_{i_1}$ and $\mca N_{i_2}$ in $\mf N$, then there exists a morphism $\varphi  \colon \mca N_{i_1} \to \mca N_{i_2}$ which induces a non-zero morphism on $H^0(\mca O_x)$ since the point $x$ is not in $(C_1 \cup C_2 )\cap C_3$. 
Thus $\mf N$ and $\mf L$ are posets and this gives the proof of the first assertion.

To finish the proof of the second assertion $(2)$, 
let us denote by $\mca N_{ij}$ the $i$-th column and the $j$-th row component of Table \ref{Summand} and put $\mf T = \{ \mca N_{ij} \}_{i\in I, j\in J}$. 
Clearly $\mf N$ is a subposet of $\mf T$. 

We see that each column subposet $\{ \mca N_{ij} \}_{i=1}^6$ is totally ordered $\{ \mca N_{1j}\leq \cdots \leq \mca N_{6j}  \}$ $(j \in \{ 1,2,3 \})$ and each row subposet is also totally ordered $\{ \mca N_{i1} \leq \mca N_{i2} \leq \mca N_{i3}  \}$ $(i \in \{  1,2,3 \})$. 
However the poset $\mf T$ is not totally ordered. 
For instance the pair $(\mca N_{31}, \mca N_{22})$ satisfies $\mca N_{31} \not\leq \mca N_{22}$ and $\mca N_{22} \not\leq \mca N_{31}$ since $\Hom_{Z_r}(\mca N_{31}, \mca N_{22}) = \Hom_{Z_r}(\mca N_{22}, \mca N_{31})=0$.  
Thus, there are at most three minimal elements in any subposet of $\mf T$. 
In particular $\mf N$ has also at most three minimal elements.  
\end{proof}
\renewcommand{\proofname}{\textit{Proof}}

\begin{remark}
Let $Z$ be the fundamental cycle of the singularity $A_n$. 
Similarly as in the proof of Lemma \ref{minimal}, a pure sheaf $\mca E$ on $Z$ is obtained from an extension on pure sheaves $\mca F$ on $C_1 + \cdots + C_{n-1}$ and $\mca G$ on $C_n$. 
The sets of direct summands of $\mca F$ and $\mca G$ are not only posets but also totally ordered sets (see also \cite[Section 6.1]{IU}). 
This is the essential difference between the singularity $A_n$ and the other Kleinian singularities. 
\end{remark}

In the rest of this note we show that the inequality in Proposition \ref{mainD} is best possible by constructing an $\mca O_X$-rigid sheaf. 
The following lemma is necessary for the construction.

\begin{lemma}\label{rigidCY}
Let $\mca D$ be a $k$-linear triangulated category. 
Suppose that $\mca D$ is CY2. 
Let $\mca F$ and $\mca G$ be in the heart $\mca A$ of a $t$-structure on $\mca D$. 
Consider an extension class $[\mca E] \in \Hom^{1}_{\mca D}(\mca G, \mca F)$
\begin{equation}
\begin{CD}
0	@>>> \mca F	@>>>	\mca E	@>>>	\mca G	@>>>	0
\end{CD} \label{nemui}
\end{equation}
such that 
\[
\Hom_{\mca D}^0(\mca F, \mca G)=\Hom^{1}_{\mca D}(\mca F, \mca F) = \Hom_{\mca D}^{1}(\mca G, \mca G)=0. 
\]
Then the following are equivalent.

\begin{enumerate}
\item[$(a)$] $\mca E$ is rigid. 
\item[$(b)$] the vector space $\Hom^{1}_{\mca D}(\mca G, \mca F)$ is generated by $\epsilon ^{\mr{L}}\big(\mr{End}_{\mca D}(\mca G)\big)$ and $\epsilon^{\mr{R}}\big(\mr{End}_{\mca D}(\mca F)\big)$ where $\epsilon = [\mca E] \in \Hom^{1}_{\mca D}(\mca G, \mca F)$ and $\epsilon^{\mr L}$ (resp. $\epsilon^{\mr R}$) is the left  (resp. right) composition:
\begin{align*}
\epsilon^{\mr{R}} \colon \mr{End}_{\mca D}(\mca F) \to \Hom^{1}_{\mca D}(\mca G, \mca F) &,\ \epsilon^{\mr{R}} (f)= f \circ \epsilon \mbox{ and } \\
\epsilon^{\mr{L}} \colon \mr{End}_{\mca D}(\mca G) \to \Hom^{1}_{\mca D}(\mca G, \mca F) &,\ \epsilon^{\mr{L}} (g)= \epsilon \circ g 
\end{align*}
\end{enumerate}
\end{lemma}

\begin{proof}
Since $\mca F$ and $\mca G$ are rigid by the assumption, we have 
\[
\hom_{\mca D}^{0}(\mca F, \mca F) = \frac{1}{2}\chi(\mca F, \mca F) \mbox{ and }\hom_{\mca D}^{0}(\mca G, \mca G) = \frac{1}{2}\chi(\mca G, \mca G). 
\]
In particular, the following is obvious since $\mca E$ is in the heart $\mca A$: 
\begin{equation}
\mca E\mbox{ is rigid} \iff \hom_{\mca D}^{0}(\mca E, \mca E) = \frac{1}{2} \chi(\mca E, \mca E).  \label{katakori}
\end{equation}

By taking $\Hom_{\mca D}(-, \mca G)$ to the sequence (\ref{nemui}), we have $\Hom_{\mca D}(\mca G, \mca G) \stackrel{\sim}{\to} \Hom_{\mca D}(\mca E, \mca G)$ and hence 
\begin{equation}
\hom_{\mca D}^0(\mca E, \mca G) = \frac{1}{2}\chi(\mca G, \mca G). \label{ccd}
\end{equation} 
Similarly we have the following exact sequence:
\[
\xymatrix{
0\ar[r]	&	\Hom_{\mca D}^{0}(\mca G, \mca F)\ar[r]	&	\Hom_{\mca D}^{0}(\mca E, \mca F)\ar[r]	&	\Hom_{\mca D}^{0}(\mca F, \mca F)\ar[lld]\\
&	\Hom_{\mca D}^{1}(\mca G, \mca F)\ar[r]	&	\Hom_{\mca D}^{1}(\mca E, \mca F)\ar[r]	&	0.	\\
}
\]
We remark that $\hom_{\mca D}^{2}(\mca G, \mca F)=\hom_{\mca D}^0(\mca F, \mca G)=0$ since $\mca D$ is CY2. 
Hence the exact sequence gives the following equation 
\begin{equation}
d_0-d_1= \frac{1}{2}\chi (\mca F, \mca F) + \chi(\mca G, \mca F),  \label{noro}
\end{equation}
where $d_i=\hom_{\mca D}^{i}(\mca E, \mca F)$ for $i \in \{  0,1\}$. 

By taking $\bb R\Hom_{\mca D}^{i}(\mca E, -)$, we have the following exact sequence:
\[
\xymatrix{
0\ar[r]	&	\Hom_{\mca D}^{0}(\mca E, \mca F)\ar[r]		&\Hom_{\mca D}^{0}(\mca E, \mca E)\ar[r]	&	\Hom_{\mca D}^{0}(\mca E, \mca G)	\ar[r]^{\delta}	&	\Hom_{\mca D}^{1}(\mca E, \mca F). \\ 
}
\]

By computation of dimensions and (\ref{ccd}), we see that the surjectivity of $\delta $ is equivalent to the following 
\begin{equation}
 \hom_{\mca D}^0(\mca E, \mca E)= d_0-d_1 + \frac{1}{2}\chi(\mca G, \mca G). \label{casio}
\end{equation}
By (\ref{noro}) the equation (\ref{casio}) is equivalent to the following:
\begin{eqnarray*}
 \hom_{\mca D}^0(\mca E, \mca E) &=& \frac{1}{2}\chi(\mca F, \mca F) + \chi(\mca G, \mca F)+\frac{1}{2}\chi(\mca G, \mca G) \\
&=& \frac{1}{2} \chi(\mca E, \mca E). 
\end{eqnarray*}
Hence $\mca E$ is rigid if and only if the morphism $\delta $ is surjective by (\ref{katakori}). 
Furthermore, the subjectivity of $\delta$ can be understood by the following diagram of exact sequences:
\[
\xymatrix{
0 \ar[r]	&	\Hom_{\mca D}^0(\mca G, \mca G) \ar[r]^{\cong}\ar[d]^{\epsilon^{\mr L}}	&	\Hom_{\mca D}^0(\mca E , \mca G) \ar[r]\ar[d]^{\delta}	&	0	\\
\Hom_{\mca D}^0(\mca F, \mca F)\ar[r]^{\epsilon^{\mr R}}	&	\Hom_{\mca D}^1(\mca G, \mca F)	\ar[r]^{\pi}	&	\Hom_{\mca D}^1(\mca E, \mca F)	\ar[r]&	0
}
\]
Since $\pi$ is surjective, 
$\delta$ is surjective if and only if $\epsilon^{ \mr{L}}$ is surjective up to the image of $\epsilon^{\mr{R}}$. 
\end{proof}

\begin{proposition}\label{bestD}
Let $Z_{4,r}$ be the reduced scheme of the fundamental cycle of $D_4$ and let $Z_r$ be the reduced scheme of the fundamental cycle of the singularity $D_n$. 
\begin{enumerate}
\item[$(1)$] There exists a rank $3$ indecomposable pure sheaf on $Z_{4,r}$ which is $\mca O_X$-rigid. 
\item[$(2)$] The inequality in Proposition \ref{mainD} is best possible. Namely there exists an $\mca O_X$-rigid pure sheaf $\mca E$ on $Z_r$ with $\rank_{Z_r}\mca E=3$. 
\end{enumerate}
\end{proposition}

\begin{proof}
We first prove the assertion (1). 
Let $\iota  \colon Z_{4,r} \to X$ be the embedding to the minimal resolution of the singularity. 
Take three pure sheaves $\mca N_{41}=\mca O_{C_1+C_3}(a_1+1, a_3+1), \mca N_{32}=\mca O_{C_2+C_3}(a_2, a_3+1)$ and $\mca N_{23} =\mca O_{C_1+C_2+C_3}(a_1+2, a_2+1, a_3 +1)$ from Table \ref{Summand} and put $\mca N =\mca N_{41}\+ \mca N_{32}\+\mca N_{23}$. 
Consider the universal extension $[\mca U] \in \Ext^1_{Z_{4,r}}(\mca O_{C_4}, \mca N)$:
\[
\begin{CD}
0 @>>> \mca N @>>> \mca U @>>> \mca O_{C_4} @>>> 0
\end{CD}. 
\]
We wish to prove that $\mca U$ is indecomposable and $\mca O_X$-rigid. 

The proof of indecomposability is essentially the same as in the proof of Theorem \ref{mainthm1}. 
It is easily see 
\begin{align}
\Hom_{Z_{4,r}}(\mca N_{41}, \mca N_{32}) &= \Hom _{Z_{4,r}}(\mca N_{32}, \mca N_{41})=0,   \label{ortho}\\ 
\Hom_{Z_{4,r}}(\mca N_{41}, \mca N_{23}) &\cong H^0(\mca O_{C_1+C_3} (1,-1)) \cong k \mbox{ and } \label{right41}	\\
\Hom_{Z_{4,r}}(\mca N_{32}, \mca N_{23}) &\cong H^0(\mca O_{C_2+C_3} (1,-1)) \cong k. 	\label{right32}
\end{align}
Take non-zero morphisms $\varphi $ and $\varphi'$ respectively in $\Hom_{Z_{4,r}}(\mca N_{41}, \mca N_{23})$ and $\Hom_{Z_{4,r}}(\mca N_{32}, \mca N_{23})$. 
Both section $\varphi$ and $\varphi'$ are zero on $C_3$ by (\ref{right41}) and (\ref{right32}). 
Moreover, since $\mca N_{41}$ is left and right orthogonal to $\mca N_{32}$ by (\ref{ortho}), the same argument in the proof of Theorem \ref{mainthm1} shows that $\mca U$ is indecomposable.

The rigidity of $\iota _*\mca U$ is a consequence of Lemma \ref{rigidCY}. 
We first show that $\mca N$ and $\mca O_{C_4}$ satisfy the assumption in Lemma \ref{rigidCY}. 
It is enough to show that $\iota _*\mca N$ is rigid.

Let us denote by $\mca D$ the derived category $D_Z(X)$. 
The rigidity of $\iota_* \mca N$ essentially follows from Riemann-Roch theorem. 
In fact, by Riemann-Roch theorem, we easily see 
\begin{align}
\chi( \iota _* \mca N_{41}, \iota _* \mca N_{32})	& =0  \label{zero} \\
\chi( \iota _* \mca N_{41}, \iota _* \mca N_{23})	& =\chi( \iota _* \mca N_{32}, \iota _* \mca N_{32})=1. \label{one}
\end{align}
By (\ref{ortho}) and (\ref{zero}) we have $\Hom_{\mca D}^1(\iota_* \mca N_{41}, \iota _* \mca N_{32})=0$. 
It is easy to see 
\begin{align}
\Hom_{Z_{4,r}}(\mca N_{23}, \mca N_{41}) &\cong H^0(\mca O_{C_1+C_3} (-1,0))=0  \mbox{ and } \label{lorth41} \\
\Hom_{Z_{4,r}}(\mca N_{23}, \mca N_{32}) &\cong H^0(\mca O_{C_2+C_3} (-1,0))=0.  \label{lorth32}
\end{align}
Hence we have $\Hom_{\mca D}^1(\iota_* \mca N_{23}, \iota _* \mca N_{41})=\Hom_{\mca D}^1(\iota_* \mca N_{23}, \iota _* \mca N_{32})=0$ by (\ref{right41}), (\ref{right32}), (\ref{one}), (\ref{lorth41}) and (\ref{lorth32}). 
Thus we see that $\iota _*\mca N$ is rigid.

By Corollary \ref{isom} the push forward $\iota _* \mca U$ is also a universal extension. 
Since the projections $p_{ij} \colon \iota _* \mca N \to \iota _* \mca N_{ij}$ to the direct summand of $\iota _* \mca N$ give a basis $\{ [ \iota _* \mca U]^{\mr{L}} (p_{ij}) \}$ of $\Hom^1_X(\iota _* \mca O_{C_4}, \iota _* \mca N)$, the push forward $\iota _*\mca U$ is rigid by Lemma \ref{rigidCY}.

For the second assertion $(2)$, let us denote by $j \colon Z_{4,r} \to Z_{r}$ the closed embedding. 
Then the push forward $j_* \mca U$ is a pure sheaf on $Z_r$ and is $\mca O_X$ rigid by Lemma \ref{extcomp}. 
Thus the second assertion holds. 
\end{proof}

\section{$\mca O_X$-rigid pure sheaves on $E_{6,7,8}$}\label{4}

\begin{proposition}\label{propE}
Let $Z$ be the fundamental cycle of the singularity $E_{n}$ for $n \in \{6,7,8  \}$ and $Z_{r}$ is the reduced scheme of $Z$. 
Then the maximal rank of $\mca O_X$-rigid  indecomposable pure sheaves is $3$:
\[
\max \{ \rank_{Z_r} \mca E   \mid  \mca E \mbox{ is an $\mca O_X$-rigid indecomposable pure sheaf } \} =3.
\]
\end{proposition}

\begin{proof}
The proof is essentially the same as in Proposition \ref{mainD}. 
Let $\mca F$ be an $\mca O_X$-rigid pure sheaf on $Z_r$. 
By the same argument in the proof of Proposition \ref{mainD}, there is a filtration of $\mca F$
\[
0 =\mca F_{0} \subset \mca F_1 \subset \mca F_2 \subset \mca F_3 \subset \mca F_4 \subset \mca F_{5} = \mca F
\]
such that 
\begin{itemize}
\item $\mca F_i/\mca F_{i-1}$ is a pure sheaf on $C_i$ for $i \in \{ 1,2,3,4  \}$ and 
\item $\mca F_5/ \mca F_4$ is a pure sheaf on $C_5 + \cdots + C_n$
\end{itemize}

The quotient $\mca F_i / \mca F_{i-1}$ is also $\mca O_X$-rigid since $\mca F$ is $\mca O_X$-rigid. 
Hence for each $i \in \{ 1,2,3,4  \}$ we have 
\[
\mca F_i/ \mca F_{i-1} \cong \mca O_{C_i}(a_i)^{\+m_i} \+ \mca O_{C_i}(a_i+1)^{\+n_i}. 
\]
Moreover we can assume that the support of each $\mca F_i$ contains $C_i$. 
Hence each direct summand of $\mca F_3$ is one of the Table \ref{F_3} and each direct summand of $\mca F_4$ supported on $C_4$ is one of Table \ref{F_4}.

\begin{table}[hbtp]
\resizebox{1\hsize}{!}{
\begin{tabular}{|c|c|c|}
\hline
$\mca O_{C_1+C2}(a_1-1, a_2)$	&$\mca O_{C_1+C_2+C_3}(a_1-1, a_2+1, a_3)$	&$\mca O_{C_1+C_2+C_3}(a_1-1, a_2+1, a_3+1)$	\\
\hline
$\mca O_{C_1+C2}(a_1, a_2)$	&$\mca O_{C_1+C_2+C_3}(a_1, a_2+1, a_3)$	&$\mca O_{C_1+C_2+C_3}(a_1, a_2+1, a_3+1)$\\
\hline
$\mca O_{C_2}(a_2)$	&$\mca O_{C_2 +C_3}(a_2+1, a_3)$				&$\mca O_{C_2+C_3}(a_2+1, a_3+1)	$\\
\hline
$\mca O_{C_1+C2}(a_1-1, a_2+1)$	&$\mca O_{C_1+C_2+C_3} (a_1 -1, a_2 +2, a_3)$&$\mca O_{C_1+C_2+C_3}(a_1 -1, a_2 +2, a_3+1)$	\\
\hline
$\mca O_{C_1+C2}(a_1, a_2+1)$	&$\mca O_{C_1+C_2+C_3}(a_1, a_2+2, a_3)$&$\mca O_{C_1+C_2+C_3}(a_1, a_2+2, a_3+1)$	\\
\hline
$\mca O_{C_2}(a_2)$	&$\mca O_{C_2+C_3}(a_2+2, a_3)$				&$\mca O_{C_2+C_3}(a_2+2, a_3+1)$		\\
\hline
\end{tabular}
}
\vspace{1mm}

\caption{}\label{F_3}
%
\resizebox{1\hsize}{!}{
\begin{tabular}{|c|c|c|}
\hline
$\mca O(a_1-1, a_2+1, a_3)$	&	$\mca O(a_1-1, a_2+1, a_3+1, a_4)$	&$\mca O(a_1-1, a_2+1, a_3+1, a_4+1)$	\\
\hline
$\mca O(a_1, a_2+1, a_3)$	&	$\mca O(a_1, a_2+1, a_3+1, a_4)$&	$\mca O(a_1, a_2+1, a_3+1, a_4+1)$\\
\hline
$\mca O(a_2+1, a_3)$				&	$\mca O(a_2+1, a_3+1,a_4)	$&	$\mca O(a_2+1, a_3+1,a_4+1)$\\
\hline
$\mca O (a_1 -1, a_2 +2, a_3)$	&	$\mca O (a_1 -1, a_2 +2, a_3+1,a_4)$&	$\mca O (a_1 -1, a_2 +2, a_3+1,a_4+1)$	\\
\hline
$\mca O(a_1, a_2+2, a_3)$	&	$\mca O(a_1, a_2+2, a_3+1,a_4)$	&	$\mca O(a_1, a_2+2, a_3+1,a_4+1)$	\\
\hline
$\mca O(a_2+2, a_3)$			&	$\mca O(a_2+2, a_3+1,a_4)$&	$\mca O(a_2+2, a_3+1,a_4+1)$		\\
\hline
$\mca O(a_3)$		&	$\mca O(a_3+1,a_4)$	&	$\mca O(a_3+1,a_4+1)$\\
\hline
$\mca O(a_1-1, a_2+1, a_3+1)$	& $\mca O(a_1-1, a_2+1, a_3+2,a_4)$& $\mca O(a_1-1, a_2+1, a_3+2,a_4+11)$\\
\hline
$\mca O(a_1, a_2+1, a_3+1)$	&	$\mca O(a_1, a_2+1, a_3+2,a_4)$&	$\mca O(a_1, a_2+1, a_3+2,a_4+1)$\\
\hline
$\mca O(a_2+1, a_3+1)$	&	$\mca O(a_2+1, a_3+2,a_4)$ &	$\mca O(a_2+1, a_3+2,a_4+1)$\\
\hline
$\mca O(a_1 -1, a_2 +2, a_3+1)$	&	$\mca O(a_1 -1, a_2 +2, a_3+2, a_4)$&	$\mca O(a_1 -1, a_2 +2, a_3+2, a_4+1)$\\
\hline
$\mca O(a_1, a_2+2, a_3+1)$	&	$\mca O(a_1, a_2+2, a_3+2,a_4)$ &	$\mathcal O(a_1, a_2+2, a_3+2,a_4+1)$\\
\hline
$\mca O(a_2+2, a_3+1)$	&	$\mca O(a_2+2, a_3+2,a_4)$ &	$\mca O(a_2+2, a_3+2,a_4+1)$\\
\hline
$\mca O(a_3+1)$	&	$\mca O(a_3+2, a_4)$&	$\mca O(a_3+2, a_4+1)$\\
\hline
\end{tabular}
}
\vspace{1mm}

\caption{For simplicity we denote by $\mca O(b_1, b_2, b_3, b_4)$ an invertible sheaf on $C_1+ C_2 +C_3 +C_4$ whose degrees are respectively $b_i$ on $C_i$. }\label{F_4}
\end{table}

Since $\mca F_5$ is pure sheaf on the tree which is isomorphic to the fundamental cycle of $A_{n-4}$, the set of direct summands of $\mca F_5/\mca F_4$ is totally ordered with respect to Definition \ref{Order}. 
Let $\mf T$ be the set of sheaves in Table \ref{F_4}. 
Similarly as in Proposition \ref{mainD} the set $\mf T$ is a poset. 
Furthermore each column set $\{  \mca N_{ij}  \}_{j=1}^3$ and each row set $\{ \mca N_{ij}  \}_{i=1}^{14}$ are totally ordered. 
Hence any subposet of $\mf T$ has at most $3$ minimal elements. 
Thus the maximum of the rank of indecomposable pure sheaves on $Z_r$ is at most $3$. 

Let $\mca U$ be a pure sheaf constructed in the proof of Proposition \ref{bestD} and $\iota \colon Z' \to Z_r$ be the closed embedding where $Z' = \sum_{j=2}^{5}C_j$. 
Then the push forward $\iota _* \mca U$ gives an indecomposable pure sheaf on $Z_r$ after the change of indexes $(1,2,3,4) \mapsto (2,4,3,5)$ .
The sheaf $\iota _* \mca U$ is $\mca O_X$-rigid by Lemma \ref{extcomp}. 
Thus the opposite inequality holds.  
\end{proof}

\renewcommand{\proofname}{\textit{Proof of Theorem \ref{mainthm2}}}
\begin{proof}
If the singularity is $D_n$ then the inequality holds and best possible by Propositions \ref{mainD} and \ref{bestD}. 
The case of $E_{6},E_{7} $ or $E_{8}$ follows form Proposition \ref{propE}. 
\end{proof}
\renewcommand{\proofname}{\textit{Proof}}

\end{document}